\documentclass{article}
\usepackage{colortbl}
\usepackage[utf8]{inputenc}
\usepackage{amsmath,stmaryrd}
\usepackage{amsthm}
\usepackage{amssymb}
\usepackage{verbatim}
\usepackage[english]{babel}
\usepackage{lipsum}
\usepackage{titleps}
\usepackage{pifont} 
\usepackage{mathrsfs}
\usepackage{mathtools}
\usepackage{footnote}
\usepackage{tikz}
\usetikzlibrary{automata,positioning} 
\usepackage{tkz-euclide} 
\usepackage{xcolor} 
\usepackage{lmodern} 
\usepackage{bbm} 
\usepackage{cancel} 
\usepackage{array}
\usepackage{multirow}
\usepackage{wrapfig,subcaption}
\usepackage[customcolors,shade]{hf-tikz} 
\usepackage{mathdots} 

\newcolumntype{L}[1]{>{\raggedright\let\newline\\\arraybackslash\hspace{0pt}}m{#1}}
\newcolumntype{C}[1]{>{\centering\let\newline\\\arraybackslash\hspace{0pt}}m{#1}}
\newcolumntype{R}[1]{>{\raggedleft\let\newline\\\arraybackslash\hspace{0pt}}m{#1}}

\newcommand{\probm}[3]{\mathbb{P}^{#1}_{#2}\left( #3 \right)}

\newcommand{\indi}[1]{\mathbbm{1}\left[ #1 \right]}

\newcommand{\bigo}[1]{O\left( #1 \right)}

\newcommand{\abs}[1]{\left\lvert #1 \right\rvert}

\newcommand{\brap}[1]{\left( #1 \right)}
\newcommand{\bras}[1]{\left[ #1 \right]}
\newcommand{\brac}[1]{\left\{ #1 \right\}}

\newcommand{\eps}{\varepsilon}
\renewcommand{\P}{\mathbb{P}}

\newtheorem{theorem}{Theorem}
\newtheorem{proposition}{Proposition}
\newtheorem{lemma}[proposition]{Lemma}

\newtheorem*{remark}{Remark}
\newtheorem{conjecture}{Conjecture}

\newcommand{\newnode}[4]{
    \node[shape=circle,inner sep=0pt, minimum size=.1cm,fill=black,label=below:{\scriptsize $#4$}] (#1) at (#2,#3) {};}

\newcommand{\newpairofnodes}[7]{
    \node[shape=circle,inner sep=0pt, minimum size=.1cm,fill=#7,label=below:{\scriptsize $\color{#7}#5$}] (0#1) at (#2,#3) {};
    \node[shape=circle,inner sep=0pt, minimum size=.1cm,fill=#7,label=above:{\scriptsize $\color{#7}#6$}] (1#1) at (#2,#4) {};}

\newcommand{\newhorizontalpaths}[5]{
    \path [-,#4] (0#1) edge[#5] node[fill=white,circle,inner sep=0.1pt] {\scriptsize $#3$} (0#2);
    \path [-,#4] (1#1) edge[#5] node[fill=white,circle,inner sep=0.1pt] {\scriptsize $#3$} (1#2);}

\newcommand{\newhorizontalpathsone}[6]{
    \path [-,#4,thick] (#5#1) edge node[fill=white,circle,inner sep=0.1pt] {\scriptsize $#3$} (#5#2);
    \path [-,#4,dashed] (#6#1) edge node[fill=white,circle,inner sep=0.1pt] {\scriptsize $#3$} (#6#2);}

\newcommand{\newverticalpath}[3]{
    \path [-,#3] (0#1) edge node[fill=white,circle,inner sep=0.1pt] {\scriptsize $#2$} (1#1);}

\newcommand{\newpath}[4]{
    \path [-,#3] (#1) edge[#4] (#2);}
    
\newcommand{\newonepath}[5]{
    \path [-,#4] (#1) edge[#5] node[fill=white,circle,inner sep=0.1pt] {\scriptsize $#3$} (#2);}

\definecolor{darkyellow}{RGB}{255,200,0}

\usepackage{hyperref}
\hypersetup{
    colorlinks=true,
    linkcolor=blue,
    filecolor=magenta,
    urlcolor=red,
}
\usepackage{cleveref}
\usepackage[top=2cm,bottom=2cm,left=2cm,right=2cm]{geometry}
\makeatletter 
\newcommand{\@minipagerestore}{
\setlength{\parindent}{0cm}
\setlength{\parskip}{5pt}}
\makeatother

\AtEndDocument{
\bigskip
\small
\par
\textsc{Tom Hutchcroft, The Division of Physics, Mathematics and Astronomy, California Institute of Technology, Pasadena, CA 91125, USA} \par
\textit{Email:} \texttt{t.hutchcroft@caltech.edu} \par

\medskip

\textsc{Petar Nizić-Nikolac and Alexander Kent, University of Cambridge, Cambridge, CB3 0WB, UK} \par
\textit{Email:} \texttt{pn310@cam.ac.uk} and \texttt{ak2147@cantab.ac.uk} \par
}

\title{The bunkbed conjecture holds in the $p\uparrow 1$ limit}
\author{Tom Hutchcroft \and Petar Nizić-Nikolac \and Alexander Kent}
\date{September 2021}

\begin{document}

\maketitle

\begin{abstract}
    Let $G=(V,E)$ be a countable graph. The Bunkbed graph of $G$ is the product graph $G \times K_2$, which has vertex set $V\times \{0,1\}$ with ``horizontal'' edges inherited from $G$ and additional ``vertical'' edges connecting $(w,0)$ and $(w,1)$ for each $w \in V$. Kasteleyn's Bunkbed conjecture states that for each $u,v \in V$ and $p\in [0,1]$, the vertex $(u,0)$ is at least as likely to be connected to $(v,0)$ as to $(v,1)$ under Bernoulli-$p$ bond percolation on the bunkbed graph. We prove that the conjecture holds in the $p \uparrow 1$ limit in the sense that for each finite graph $G$ there exists $\varepsilon(G)>0$ such that the bunkbed conjecture holds for $p \geqslant 1-\eps(G)$.
\end{abstract}

\section{Introduction}

In \textbf{Bernoulli bond percolation}, the edges of a countable graph $G=(V,E)$ (which we allow to contain self-loops and/or multiple edges) are each deleted or retained  independently at random with retention probability $p\in [0,1]$. We call retained edges \textbf{open}, deleted edges \textbf{closed}, and write $\P_p=\P^G_p$ for the law of the resulting random subgraph. Percolation theorists seek to understand the geometry of the connected components of the random subgraph that remains, and how this geometry depends on the parameter $p$. Despite the simplicity of the model, this is a rich subject with many connections to other topics in mathematics, physics, and computer science; see e.g.\ \cite{MR3751350,MR1707339} for background and overview.

When studying percolation, it is often important to understand the behaviour of connection probabilities between vertices, also known as the \textbf{two-point function}. Situations often arise in which it is easier to bound \emph{averages} of the two-point function (e.g.\ over a box in the hypercubic lattice) than it is to prove pointwise bounds; see for instance the thirteen-year gap between Hara and Slade's proof of mean-field behaviour for the \emph{Fourier transform} of the two-point function for high-dimensional percolation in 1990 \cite{MR1043524} and the subsequent proof of pointwise bounds by Hara, van der Hofstad, and Slade in 2003 \cite{MR1959796}.  As such, it would be very useful to have general techniques to convert averaged bounds into pointwise bounds. Such a conversion would be straightforward if one knew that the connection probability from the origin to $x=(x_1,\ldots,x_d)$ on the hypercubic lattice $\mathbb{Z}^d$ is a decreasing function of $|x_i|$ for each $i=1,\ldots,d$. Unfortunately, this intuitively plausible statement seems to be out of reach of present methods, although an analogous statement is known to hold for the Ising model \cite{MR676312}.

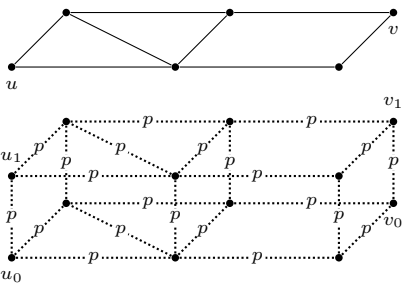
\begin{wrapfigure}{l}{0.35\textwidth}
    \vspace{-0.7cm}
    \begin{center}
    \begin{tikzpicture}[scale=0.725]
        \def\y{3.5}
        
        \newnode{a}{0}{0+\y}{u}
        \newnode{b}{3}{0+\y}{}
        \newnode{c}{6}{0+\y}{}
        \newnode{d}{1}{1+\y}{}
        \newnode{e}{4}{1+\y}{}
        \newnode{f}{7}{1+\y}{v}
        
        \newpath{a}{b}{}{}
        \newpath{b}{c}{}{}
        \newpath{c}{f}{}{}
        \newpath{d}{e}{}{}
        \newpath{e}{f}{}{}
        \newpath{a}{d}{}{}
        \newpath{b}{d}{}{}
        \newpath{b}{e}{}{}
        
        \newpairofnodes{a}{0}{0}{1.5}{u_0}{u_1}{black}
        \newpairofnodes{b}{3}{0}{1.5}{}{}{black}
        \newpairofnodes{c}{6}{0}{1.5}{}{}{black}
        \newpairofnodes{d}{1}{1}{2.5}{}{}{black}
        \newpairofnodes{e}{4}{1}{2.5}{}{}{black}
        \newpairofnodes{f}{7}{1}{2.5}{v_0}{v_1}{black}
        
        \newverticalpath{a}{p}{thick, densely dotted}
        \newverticalpath{b}{p}{thick, densely dotted}
        \newverticalpath{c}{p}{thick, densely dotted}
        \newverticalpath{d}{p}{thick, densely dotted}
        \newverticalpath{e}{p}{thick, densely dotted}
        \newverticalpath{f}{p}{thick, densely dotted}
        
        \newhorizontalpaths{a}{b}{p}{thick, densely dotted}{}
        \newhorizontalpaths{b}{c}{p}{thick, densely dotted}{}
        \newhorizontalpaths{c}{f}{p}{thick, densely dotted}{}
        \newhorizontalpaths{d}{e}{p}{thick, densely dotted}{}
        \newhorizontalpaths{e}{f}{p}{thick, densely dotted}{}
        \newhorizontalpaths{a}{d}{p}{thick, densely dotted}{}
        \newhorizontalpaths{b}{d}{p}{thick, densely dotted}{}
        \newhorizontalpaths{b}{e}{p}{thick, densely dotted}{}
    \end{tikzpicture}
    \caption{A graph $G$ (top) and the bunkbed model on $G$ (bottom).}
    \end{center}
    \vspace{-1.0cm}
\end{wrapfigure}

These considerations lend motivation to the \emph{bunkbed conjecture}, which was first formulated by Kasteleyn in the mid 1980s \cite[Remark 5]{MR1825144}. Given a graph $G=(V,E)$, the \textbf{bunkbed graph} is defined to be the Cartesian product $G\times K_2$ of $G$ and the graph with one edge (sometimes known as the box product and denoted $G \boxempty K_2$), which has vertex set $V\times \{0,1\}$ with ``horizontal'' edges inherited from $G$ and additional ``vertical'' edges connecting $(w,0)$ and $(w,1)$ for each $w \in V$. We refer to Bernoulli percolation on $G \times K_2$ as the \textbf{bunkbed model} and write $\mathbb{P}_p^\mathrm{bb}=\P_p^{G\times K_2}$ for its law. To lighten notation, we will often write $v_i=(v,i)$ for $v\in V$ and $i\in \{0,1\}$.

\begin{conjecture}[Bunkbed conjecture]\label{conj:bunkbed}
    Let $G=(V,E)$ be a finite graph. Then $\probm{\mathrm{bb}}{p}{u_0 \leftrightarrow v_0} \geqslant \probm{\mathrm{bb}}{p}{u_0 \leftrightarrow v_1}$ for every $p \in \bras{0,1}$ and $u,v\in V$.
\end{conjecture}

(In fact Kastelyn may have conjectured the stronger statement that this inequality remains true after conditioning on which vertical edges are present; see Theorem \ref{thm:mainprime} below.)
Despite its simple and intuitive statement, and proofs of analogous statements for the Ising model and simple random walk by H\"aggstr\"om \cite{MR1680084,haggstrom2003probability},  the bunkbed conjecture has so far evaded resolution. Indeed, even the case in which $G$ is the complete graph was verified only in the recent work of Van Hintum and Lammers \cite{MR3886521} following partial results of De Buyer \cite{deBuyer2016,deBuyer2018}. Further special cases of the conjecture have been studied by Leander \cite{Leander} and Linusson\footnote{Parts of the work of Linusson, including his analysis of outerplanar graphs, were unfortunately subject to a serious error as detailed in the erratum \cite{MR4015661}.}~\cite{MR2745680}, and interesting related correlation inequalities inspired by the bunkbed conjecture have been developed in the work of Van den Berg, H\"aggstr\"om, and Kahn \cite{MR2268229,MR1825144}.

The goal of this paper is to prove that the bunkbed conjecture holds (with strict inequality) in the limit as $p \uparrow 1$. We have not attempted to optimize the resulting constants.

\begin{theorem}\label{thm:main}
    Let $G=(V,E)$ be a finite, connected graph. Then $\probm{\mathrm{bb}}{p}{u_0 \leftrightarrow v_0} > \probm{\mathrm{bb}}{p}{u_0 \leftrightarrow v_1}$ for every $u,v\in V$ and $1-2^{-\abs{E}/2-2} \leqslant p<1$.
\end{theorem}

The complementary fact that the conjecture holds in the limit as $p \downarrow 0$ holds for trivial reasons. Indeed, if we consider Bernoulli bond percolation between any two vertices $u$ and $v$ on a finite graph $G=(V,E)$, then one can easily see that
\begin{equation}
    \mathbb{P}^G_p(u \leftrightarrow v) = \sum_{\omega \in \Omega} \indi{u \overset{\omega}{\leftrightarrow} v} p^{\#\text{open}}\brap{1-p}^{\#\text{closed}} = \#\{\text{geodesics from $u$ to $v$}\} \cdot p^{d(u,v)} \pm \bigo{p^{d\brap{u, v}+1}} 
\end{equation}
where $d(u,v)$ is the graph distance between $u$ and $v$, i.e., the length of the shortest path connecting $u$ to~$v$. Since the graph distances in the bunkbed graph satisfy $d(u_0,v_1)=d(u_0,v_0)+1$ for every $u,v\in V$, it follows that the lowest order non-zero coefficient of $p$ in $\probm{\mathrm{bb}}{p}{u_0 \leftrightarrow v_0}$ occurs earlier than in $\probm{\mathrm{bb}}{p}{u_0 \leftrightarrow v_1}$, so that $\probm{\mathrm{bb}}{p}{u_0 \leftrightarrow v_1} \ll \probm{\mathrm{bb}}{p}{u_0 \leftrightarrow v_0}$ as $p\downarrow 0$ and the conjecture holds in the $p\downarrow 0$ limit as claimed. We show in Section~\ref{subsec:example} that Theorem~\ref{thm:main} cannot be deduced by a similar perturbative analysis, since it is possible for $\abs{\probm{\mathrm{bb}}{p}{u_0 \nleftrightarrow v_0}-\probm{\mathrm{bb}}{p}{u_0 \nleftrightarrow v_1}}$ to be much smaller than any power of $\probm{\mathrm{bb}}{p}{u_0 \nleftrightarrow v_0}$ or $\probm{\mathrm{bb}}{p}{u_0 \nleftrightarrow v_1}$ as $p\uparrow 1$.

Note that if one could find a single value of $p$ such that the bunkbed conjecture held for \emph{every} finite graph at $p$, then the bunkbed conjecture would hold for every graph $G$ and every  $p$ by a theorem of Rudzinski and Smyth \cite{rudzinski2016equivalent}.

\medskip

\textbf{Acknowledgments.} This paper is the result of an undergraduate summer research project at the University of Cambridge in the summer of 2020, where PNN and AK were mentored by TH. PNN was supported jointly by a Trinity College Summer Studentship (F. J. Woods Fund) and a CMS Summer Studentship, AK was supported by a CMS Summer Research in Mathematics bursary, and TH was supported in part by ERC starting grant 804166 (SPRS). We thank Piet Lammers for helpful comments on a draft.

\subsection{Analysis of an example}\label{subsec:example}

\begin{wrapfigure}{l}{0.35\textwidth}
    \vspace{-0.8cm}
    \begin{center}
    \begin{tikzpicture}[scale=0.80]
        \newpairofnodes{a}{0}{0}{1.5}{{(0,0)}}{{(1,0)}}{black}
        \newpairofnodes{b}{2}{0}{1.5}{{(0,1)}}{{(1,1)}}{black}
        \newverticalpath{a}{p}{thick, densely dotted}
        \newverticalpath{b}{p}{thick, densely dotted}
        \newhorizontalpaths{a}{b}{p}{thick, densely dotted}{}
        
        \def\y{1.5}
        \newpairofnodes{c0}{2.5}{0+0.333*\y}{1.5+0.333*\y}{}{}{white}
        \newhorizontalpaths{b}{c0}{}{thick, densely dotted}{}
        \newpairofnodes{d0}{3}{0+0.667*\y}{1.5+0.667*\y}{}{}{white}
        \newpairofnodes{e0}{3.5}{0+\y}{1.5+\y}{{(n-1,0)\text{\ding{51}}}}{{(n-1,1)\text{\ding{55}}}}{black}
        \newpairofnodes{f0}{5.5}{0+\y}{1.5+\y}{{(n,0)\text{\ding{51}}}}{{(n,1)\text{\ding{55}}}}{black}
        \newverticalpath{e0}{p}{thick, densely dotted}
        \newverticalpath{f0}{p}{thick, densely dotted, red}
        \newhorizontalpaths{d0}{e0}{}{thick, densely dotted}{}
        \newonepath{0e0}{0f0}{p}{thick, densely dotted, green}{}
        \newonepath{1e0}{1f0}{p}{thick, densely dotted, black}{}
        \node[rotate=90] at (6.2,0.75+\y) {$\tilde{A}_{n-1}$ case};
        \node at (4.5,0.75+\y) {\scriptsize $q(1-q)$};
        \node at (2.75,0.75+0.5*\y) {$\iddots$};
        
        \def\y{-1.5}
        \newpairofnodes{c1}{2.5}{0+0.333*\y}{1.5+0.333*\y}{}{}{white}
        \newhorizontalpaths{b}{c1}{}{thick, densely dotted}{}
        \newpairofnodes{d1}{3}{0+0.667*\y}{1.5+0.667*\y}{}{}{white}
        \newpairofnodes{e1}{3.5}{0+\y}{1.5+\y}{{(n-1,0)\text{\ding{51}}}}{{(n-1,1)\text{\ding{51}}}}{black}
        \newpairofnodes{f1}{5.5}{0+\y}{1.5+\y}{{(n,0)\text{\ding{51}}}}{{(n,1)\text{\ding{55}}}}{black}
        \newverticalpath{e1}{p}{thick, densely dotted}
        \newverticalpath{f1}{p}{thick, densely dotted, red}
        \newhorizontalpaths{d1}{e1}{}{thick, densely dotted}{}
        \newonepath{0e1}{0f1}{p}{thick, densely dotted, green}{}
        \newonepath{1e1}{1f1}{p}{thick, densely dotted, red}{}
        \node[rotate=90] at (6.2,0.75+\y) {$C_{n-1}$ case};
        \node at (4.5,0.75+\y) {\scriptsize $q^2(1-q)$};
        \node at (2.75,0.75+0.5*\y) {$\ddots$};
        
        \path (2.75,0.75) edge (6.25,0.75);
    \end{tikzpicture}
    \vspace{-0.2cm}
    \captionof{figure}{Computation of $\tilde{A}_n$.}
    \label{fig:line}
    \end{center}
    \vspace{-1.5cm}
\end{wrapfigure}
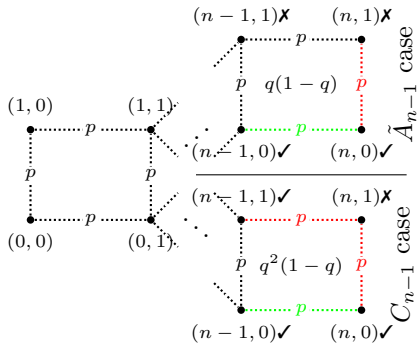

Before proving Theorem~\ref{thm:main}, we first consider the following illustrative example which shows that Theorem~\ref{thm:main} does not follow by naive perturbation arguments. For each $n\geq 0$ let $G_n$ be the path of length $n$ with vertices $\{0,\ldots,n\}$, let $L_n=G_n \times K_2$ be the associated bunkbed graph, let $p=1-q\in (0,1)$, and consider the probabilities
\begin{align*}
    &A_n=A_n(q) = \probm{L_n}{1-q}{(0,0) \leftrightarrow (n,0)},\\
    &B_n=B_n(q) = \probm{L_n}{1-q}{(0,0) \leftrightarrow (n,1)},\\
    &C_n=C_n(q)=\probm{L_n}{1-q}{(0,0) \leftrightarrow (n,0) \leftrightarrow (n,1)}.
\end{align*}
If we also consider the probabilities
\begin{align*}
    &\tilde{A}_n=\tilde{A}_n(q) = \probm{L_n}{1-q}{(0,0) \leftrightarrow (n,0) \not\leftrightarrow (n,1)},\\
    &\tilde{B}_n=\tilde{B}_n(q) = \probm{L_n}{1-q}{(0,0) \leftrightarrow (n,1) \not\leftrightarrow (n,0)}.
\end{align*}
then we easily see by considering the copy of $L_{n-1}$ inside $L_n$ as in Figure \ref{fig:line} that
\begin{multline}
    \tilde{A}_n =  \tilde{A}_{n-1} \cdot \probm{L_{n}}{1-q}{\{(n-1,0) , (n,0)\} \text{ open}, \{(n,0),(n,1)\} \text{ closed}} \\+ C_{n-1} \cdot \probm{L_{n}}{1-q}{\{(n-1,0) , (n,0)\} \text{ open}, \{(n,0),(n,1)\} \text{ and } \{(n-1,1),(n,1)\} \text{ closed}} \\= q(1-q) \tilde A_{n-1}+q^2(1-q)C_{n-1}
\end{multline}
for every $n\geq 1$ and similarly that
\begin{equation}
    \tilde{B}_n  = q(1-q) \tilde B_{n-1} + q^2(1-q)C_{n-1}
\end{equation}
for every $n\geq 1$.
Taking the difference of these two equations, the $C_{n-1}$ terms cancel to yield that
\begin{equation}
    A_n - B_n = \tilde{A}_n - \tilde{B}_n = q(1-q) \cdot \bras{\tilde{A}_{n-1} - \tilde{B}_{n-1}} = q(1-q) \cdot \bras{A_{n-1} - B_{n-1}}
\end{equation}
for every $n\geq 1$, and since $A_0 = 1$ and $B_0 = 1-q$ we conclude by induction that
\begin{equation}\label{eq:gap}
    A_{n} - B_{n} = q^{(n+1)}\brap{1-q}^n
\end{equation}
for every $n\geq 0$. 
Thus, the difference between the two connection probabilities is very small when $n$ is large, uniformly in the choice of $q$. In fact one can say rather more than this.

Indeed, if we define $D_n=D_n(q)=\tilde A_n + \tilde B_n$ to be the probability that $(0,0)$ is connected to exactly one of $(n,0)$ and $(n,1)$ in $L_n$ then similar reasoning to above yields that
\[
\begin{pmatrix}
C_{0}\\D_{0}
\end{pmatrix}
=
\begin{pmatrix}
1-q\\q
\end{pmatrix}
\qquad \text{ and } \qquad
\begin{pmatrix}
C_{n}\\D_{n}
\end{pmatrix}=
\begin{pmatrix}
(1+2q)(1-q)^2 & (1-q)^2  \\
2q^2(1-q) & q(1-q)
\end{pmatrix}
\begin{pmatrix}
C_{n-1}\\D_{n-1}
\end{pmatrix}
\qquad \text{ for each $n\geq 1$}
\]
and solving this linear recursion leads to the expansions
\[
\begin{pmatrix}
C_{n}\\D_{n}
\end{pmatrix} 
= 
\left(\begin{array}{rrrrrrr}
1&-(n+3)q^2&-2(n-1)q^3 &+ \frac{1}{2}(n^2+9n+16)q^4 &+2(n^2+3n-10)q^5 & \cdots &\pm O(q^{n+1})\\
&2q^2 &&-2(n+3)q^4 &-4(n-2)q^5 & \cdots &\pm O(q^{n+1})
\end{array}
\right).
\]
Using the relationship $A_n+B_n=2C_n+D_n$ (which follows by linearity of expectation) together with the equality $A_n-B_n=q^{n+1}(1-q)^n$ above yields the expansions of $A_n$ and $B_n$ both share the same first $n+1$ terms
\begin{equation}
    1 - (n+2)q^2 - 2(n-1)q^3 + \frac{1}{2}(n^2+7n+10)q^4+2(n^2+2n-8)q^5 - \cdots.
\end{equation}
In fact, taking the same computations further one arrives at the simple exact limiting formula
\begin{equation}
    \lim_{n\to\infty} A_n\left(\frac{\lambda}{\sqrt{n}}\right) = \lim_{n\to\infty} B_n\left(\frac{\lambda}{\sqrt{n}}\right) = e^{-\lambda^2} = 1 - \lambda^2 + \frac{\lambda^4}{2} - \cdots
\end{equation}
for each $\lambda>0$. Meanwhile, the difference between $A_n(\lambda/\sqrt{n})$ and $B_n(\lambda/\sqrt{n})$ goes to zero superexponentially fast as $n \to \infty$ for each fixed $\lambda>0$ by \eqref{eq:gap}.
Thus, in this example the bunkbed inequality $A_n(q)\geqslant B_n(q)$ is completely invisible when expanding $A_n$ and $B_n$ to any finite order around $q=0$.

To prove Theorem~\ref{thm:main} we will derive an alternative expansion of connection probabilities in the bunkbed graph that allows us to cancel off most of the contribution to each of the probabilities $\P^\mathrm{bb}_p(u_0 \leftrightarrow v_0)$ and $\P^\mathrm{bb}_p(u_0 \leftrightarrow v_1)$ and be left with a quantity we can show is positive for $p$ sufficiently close to $1$.

\section{Proof of the main result}

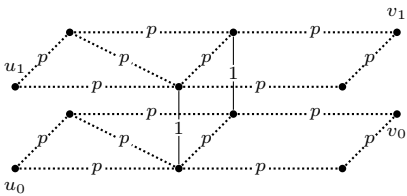
\begin{wrapfigure}{l}{0.35\textwidth}
    \vspace{-0.8cm}
    \begin{center}
    \begin{tikzpicture}[scale=0.725]
        \newpairofnodes{a}{0}{0}{1.5}{u_0}{u_1}{black}
        \newpairofnodes{b}{3}{0}{1.5}{}{}{black}
        \newpairofnodes{c}{6}{0}{1.5}{}{}{black}
        \newpairofnodes{d}{1}{1}{2.5}{}{}{black}
        \newpairofnodes{e}{4}{1}{2.5}{}{}{black}
        \newpairofnodes{f}{7}{1}{2.5}{v_0}{v_1}{black}
        
        \newverticalpath{b}{1}{}
        \newverticalpath{e}{1}{}
        
        \newhorizontalpaths{a}{b}{p}{thick, densely dotted}{}
        \newhorizontalpaths{b}{c}{p}{thick, densely dotted}{}
        \newhorizontalpaths{c}{f}{p}{thick, densely dotted}{}
        \newhorizontalpaths{d}{e}{p}{thick, densely dotted}{}
        \newhorizontalpaths{e}{f}{p}{thick, densely dotted}{}
        \newhorizontalpaths{a}{d}{p}{thick, densely dotted}{}
        \newhorizontalpaths{b}{d}{p}{thick, densely dotted}{}
        \newhorizontalpaths{b}{e}{p}{thick, densely dotted}{}
    \end{tikzpicture}
    \vspace{-0.2cm}
    \captionof{figure}{The bunkbed model after conditioning on vertical edges.}
    \label{fig:modelconditioning}
    \end{center}
    \vspace{-1.0cm}
\end{wrapfigure}

We now prove Theorem~\ref{thm:main}. In fact we will prove a stronger version of the theorem which allows us to condition on which vertical edges are open. We write $T$ for the set of $w\in V$ such that the vertical edge connecting $w_0$ and $w_1$ is open. One can deduce Theorem~\ref{thm:main} from the following theorem by taking expectations over $T$.

\begin{theorem}\label{thm:mainprime}
    Let $G=(V,E)$ be a finite graph. Then $\probm{\mathrm{bb}}{p}{u_0 \leftrightarrow v_0 \mid T=t} \geqslant \probm{\mathrm{bb}}{p}{u_0 \leftrightarrow v_1 \mid T=t}$ for every $u,v\in V$, $t \subseteq V$, and $1-2^{-\abs{E}/2-2}\leqslant p<1$, with strict inequality if and only if $u$ and $v$ are connected in $V \setminus t$.
\end{theorem}

Here, the condition that $u$ and $v$ are connected in $V \setminus t$ includes the condition that $u,v \notin t$.

The proof of Theorem~\ref{thm:mainprime} will rely on a relationship between percolation on the bunkbed graph and another model due to Linusson \cite{MR2745680}, which we call the \textbf{alternative bunkbed model}. 

\begin{wrapfigure}{l}{0.35\textwidth}
    \vspace{-0.4cm}
    \begin{center}
    \begin{tikzpicture}[scale=0.725]
        \def\y{3.5}
        
        \newpairofnodes{a}{0}{0+\y}{1.5+\y}{u_0}{u_1}{black}
        \newpairofnodes{b}{3}{0+\y}{1.5+\y}{}{}{black}
        \newpairofnodes{c}{6}{0+\y}{1.5+\y}{}{}{black}
        \newpairofnodes{d}{1}{1+\y}{2.5+\y}{}{}{black}
        \newpairofnodes{e}{4}{1+\y}{2.5+\y}{}{}{black}
        \newpairofnodes{f}{7}{1+\y}{2.5+\y}{v_0}{v_1}{black}
        
        \newverticalpath{b}{1}{}
        \newverticalpath{e}{1}{}
        
        \newhorizontalpaths{a}{b}{p}{thick, densely dotted, darkyellow}{}
        \newhorizontalpaths{b}{c}{p}{thick, densely dotted, darkyellow}{}
        \newhorizontalpaths{c}{f}{p}{thick, densely dotted, green}{}
        \newhorizontalpaths{c}{f}{}{dashed, bend left, ->}{}
        \newhorizontalpaths{d}{e}{p}{thick, densely dotted, darkyellow}{}
        \newhorizontalpaths{e}{f}{p}{thick, densely dotted, darkyellow}{}
        \newhorizontalpaths{a}{d}{p}{thick, densely dotted, green}{}
        \newhorizontalpaths{d}{a}{}{dashed, bend right, ->}{}
        \newhorizontalpaths{b}{d}{p}{thick, densely dotted, darkyellow}{}
        \newhorizontalpaths{b}{e}{p}{thick, densely dotted, red}{}
        \node at (3.25,0.25+\y) {\scriptsize $\times$};
        \node at (3.75,0.75+\y) {\scriptsize $\times$};
        \node at (3.25,1.75+\y) {\scriptsize $\times$};
        \node at (3.75,2.25+\y) {\scriptsize $\times$};
        
        \newpairofnodes{a1}{0}{0}{1.5}{u_0}{u_1}{black}
        \newpairofnodes{b1}{3}{0}{1.5}{}{}{black}
        \newpairofnodes{e1}{4}{1}{2.5}{}{}{black}
        \newpairofnodes{f1}{7}{1}{2.5}{v_0}{v_1}{black}
        
        \newverticalpath{b1}{1}{}
        \newverticalpath{e1}{1}{}
        
        \newhorizontalpaths{a1}{b1}{1/2}{thick, densely dotted}{}
        \newhorizontalpaths{a1}{b1}{1/2}{thick, densely dotted, bend right}{}
        \newhorizontalpaths{a1}{e1}{1/2}{thick, densely dotted}{}
        \newhorizontalpaths{b1}{f1}{1/2}{thick, densely dotted}{}
        \newhorizontalpaths{e1}{f1}{1/2}{thick, densely dotted}{}
    \end{tikzpicture}
    \vspace{-0.2cm}
    \captionof{figure}{A tripartition $s = \brap{\color{red}s_0,\color{darkyellow}s_1,\color{green}s_2\color{black}}$ (top) and its induced alternative bunkbed model (bottom).}
    \label{fig:modelconditioning}
    \end{center}
    \vspace{-1.0cm}
\end{wrapfigure}
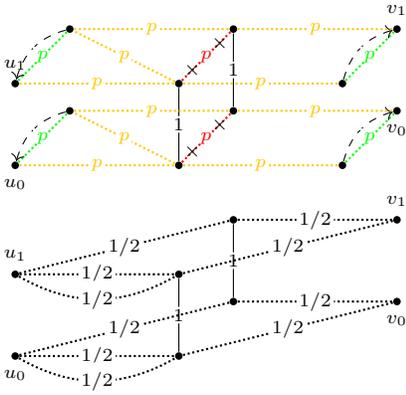

Let $G=(V,E)$ be a finite graph, let $t \subseteq V$ be a set of vertices, and let $p\in (0,1)$. If we condition on the event $T=t$, the horizontal edges of the bunkbed graph are still i.i.d., each with probability $p$ of being open.

 We call a partition of $E$ into three disjoint (possibly empty) sets a \textbf{tripartition} of $E$, and write $\mathcal{S}$ for the set of tripartitions of $E$.
Let $S=(S_0,S_1,S_2)$ be the random tripartition of the edge set of $G$ into those edges that have zero, one, or two corresponding open horizontal edges in the bunkbed graph. We can sample from the conditional distribution of percolation on the bunkbed graph given $T=t$ and $S=s=(s_0,s_1,s_2)$ by choosing for each edge $e\in s_1$ exactly one of the two corresponding horizontal bunkbed edges to be open, each with probability $1/2$, independently at random for each $e\in s_1$; if $e \in s_0$ then both corresponding horizontal bunkbed edges are closed while if $e \in s_2$ then both corresponding bunkbed edges are open. 

Given a tripartition $s=(s_0,s_1,s_2)$ of the edge set of $G$, we write $G(s)=(V(s),E(s))$ for the graph formed by deleting the edges of $s_0$ and contracting every edge of $s_2$, and write $\pi_s:V\to V(s)$ for the resulting projection map (where each vertex of $v$ is identified with its connected component in $s_2$). Note that $G(s)$ may contain self-loops and multiple edges even if the original graph $G$ did not.

These considerations lead to the definition of the alternative bunkbed model, which we now introduce. Let $G=(V,E)$ be a finite graph and let $t\subseteq V$ be a distinguished set of vertices. Let $\mathbb{Q}_G^t$ be the law of the random subgraph of the bunkbed graph $G \times K_2$ in which we include a vertical edge $\{u_0,u_1\}$ if and only if $u\in t$, and for each horizontal edge of $G$ choose exactly one of the two corresponding horizontal edges of $G\times K_2$ to be open, independently at random with probability $1/2$ each. Given a random variable with law $\mathbb{Q}_G^t$, we say that an edge $e$ of $G$ is \textbf{up} if the upper of its two corresponding horizontal edges is open, and \textbf{down} otherwise.
    
The description of the conditional distribution of the bunkbed model given $T=t$ and $S=s$ discussed above leads to the identity
\begin{multline}\label{eq:Linusson_expansion}
    \probm{\mathrm{bb}}{p}{u_0 \leftrightarrow v_0 \mid T=t} - \probm{\mathrm{bb}}{p}{u_0 \leftrightarrow v_1 \mid T=t} \\= \sum_{s \in \mathcal{S}}  
    \left[\mathbb{Q}_{G(s)}^{\pi_s(t)}\left(\pi_s(u)_0 \leftrightarrow \pi_s(v)_0\right) - \mathbb{Q}_{G(s)}^{\pi_s(t)}\left(\pi_s(u)_0 \leftrightarrow\pi_s(v)_1\right)\right]\probm{\mathrm{bb}}{p}{S=s}.
\end{multline}
Note also that we have the equality
\begin{equation}
    \probm{\mathrm{bb}}{p}{S=s}  = (1-p)^{2\abs{s_0}} (2p(1-p))^{\abs{s_1}} p^{2\abs{s_2}} 
\end{equation}
for each tripartition $s\in \mathcal{S}$. Linusson has conjectured \cite[Conjecture 2.5]{MR2745680} that 
\begin{equation}
    \mathbb{Q}_{G}^{t}(u_0 \leftrightarrow v_0) \geqslant \mathbb{Q}_{G}^{t}(u_0 \leftrightarrow v_1)
\end{equation}
for every finite graph $G=(V,E)$, $t\subseteq V$, and $u,v \in V$; it follows from \eqref{eq:Linusson_expansion} that this conjecture implies Conjecture~\ref{conj:bunkbed}. We will instead use \eqref{eq:Linusson_expansion} to probe the validity of the bunkbed conjecture when $p$ is close to $1$, the main idea being that $G(s)$ is typically very small in this regime since most edges will be contracted.

Fix a finite graph $G=(V,E)$, a set $t\subseteq V$, two vertices $u,v\in V$ and a tripartition $s\in \mathcal{S}$. We say that a path in $G(s)$ from $\pi_s(u)$ to $\pi_s(v)$ is \emph{vertical-free} if all vertices on that path (including the endpoints) are not in $\pi_s(t)$. Let $d(s)$ be the length of the shortest such path if it exists, and $\infty$ otherwise.
We also define
\[
F(s) \coloneqq  \mathbb{Q}_{G(s)}^{\pi_s(t)}\left(\pi_s(u)_0 \leftrightarrow \pi_s(v)_0\right) - \mathbb{Q}_{G(s)}^{\pi_s(t)}\left(\pi_s(u)_0 \leftrightarrow\pi_s(v)_1\right).
\]
The idea is to partition $\mathcal{S}$ into several classes according to the value of $d(s)$ as follows:
\colorlet{lightyellow}{yellow!35!}
\colorlet{lightgreen}{green!20!}
\colorlet{lightred}{red!20!}
\hfsetbordercolor{white}
\[\hfsetfillcolor{lightyellow}\tikzmarkin{d1}(0.1,-0.2)(-0.1,0.4) \mathcal{S}_{\infty} = \brac{s \in \mathcal{S} \colon d(s) = \infty} \tikzmarkend{d1} \qquad
\hfsetfillcolor{lightgreen}\tikzmarkin{d2}(0.1,-0.2)(-0.1,0.4) \mathcal{S}_{\leqslant 1} = \brac{s \in \mathcal{S} \colon d(s) \leqslant 1} \tikzmarkend{d2} \qquad
\hfsetfillcolor{lightred}\tikzmarkin{d3}(0.1,-0.2)(-0.1,0.4) \mathcal{S}_{\geqslant 2} = \brac{s \in \mathcal{S} \colon d(s) \geqslant 2} \tikzmarkend{d3}.\]
 We will show that $F$ vanishes on $\mathcal{S}_{\infty}$ and that it is uniformly positive on $\mathcal{S}_{\leqslant 1}$. The contribution of the final set $\mathcal{S}_{\geqslant 2}$ will be controlled using the trivial inequality $F(s) \geqslant -2$.

We recall the following argument of Linusson \cite[Lemma 2.3]{MR2745680}, which we include a proof of for completeness.

\begin{wrapfigure}{l}{0.45\textwidth}
    \vspace{-0.4cm}
    \begin{center}
    \begin{tikzpicture}[scale=0.7]
        \newpairofnodes{a}{0}{0}{1.5}{u_0}{u_1}{black}
        \newpairofnodes{b}{3}{0}{1.5}{}{}{black}
        \newpairofnodes{e}{4}{1}{2.5}{}{}{black}
        \newverticalpath{b}{1}{}
        \newverticalpath{e}{1}{}
        \newhorizontalpathsone{a}{b}{1/2}{blue}{0}{1}
        \newhorizontalpathsone{a}{b}{1/2}{blue, bend right}{0}{1}
        \newhorizontalpathsone{a}{e}{1/2}{blue}{0}{1}
        \draw[thick, dotted, fill=green, fill opacity=0.3] (2.8, -0.4) -- (4.2, 0.9) -- (4.2, 2.9) -- (2.8, 1.6) -- cycle;
        
        \def\x{2.5} \def\y{-1.4}
        \newpairofnodes{b1}{3+\x}{0+\y}{1.5+\y}{}{}{black}
        \newpairofnodes{e1}{4+\x}{1+\y}{2.5+\y}{}{}{black}
        \newpairofnodes{f1}{7+\x}{1+\y}{2.5+\y}{v_0}{v_1}{black}
        \newverticalpath{b1}{1}{}
        \newverticalpath{e1}{1}{}
        \newhorizontalpathsone{e1}{f1}{1/2}{red}{0}{1}
        \newhorizontalpathsone{b1}{f1}{1/2}{red}{0}{1}
        \path [->, green, ultra thick, bend right] (3.5,0.45) edge node[below, sloped] {\small $f_{uv} = 1$} (3.5+\x,0.45+\y);
        \draw[thick, dotted, fill=green, fill opacity=0.3] (2.8+\x, -0.4+\y) -- (4.2+\x, 0.9+\y) -- (4.2+\x, 2.9+\y) -- (2.8+\x, 1.6+\y) -- cycle;
        
        \def\x{2.5} \def\y{1.4}
        \newpairofnodes{b1}{3+\x}{0+\y}{1.5+\y}{}{}{black}
        \newpairofnodes{e1}{4+\x}{1+\y}{2.5+\y}{}{}{black}
        \newpairofnodes{f1}{7+\x}{1+\y}{2.5+\y}{v_0}{v_1}{black}
        \newverticalpath{b1}{1}{}
        \newverticalpath{e1}{1}{}
        \newhorizontalpathsone{e1}{f1}{1/2}{red}{1}{0}
        \newhorizontalpathsone{b1}{f1}{1/2}{red}{1}{0}
        \path [->, green, ultra thick, bend left] (3.5,2.05) edge node[above, sloped] {\small $f_{uv} = -1$} (3.5+\x,2.05+\y);
        \draw[thick, dotted, fill=green, fill opacity=0.3] (2.8+\x, -0.4+\y) -- (4.2+\x, 0.9+\y) -- (4.2+\x, 2.9+\y) -- (2.8+\x, 1.6+\y) -- cycle;
    \end{tikzpicture}
    \vspace{-0.2cm}
    \captionof{figure}{The mirroring argument.}
    \label{fig:mirroring}
    \end{center}
    \vspace{-1.5cm}
\end{wrapfigure}
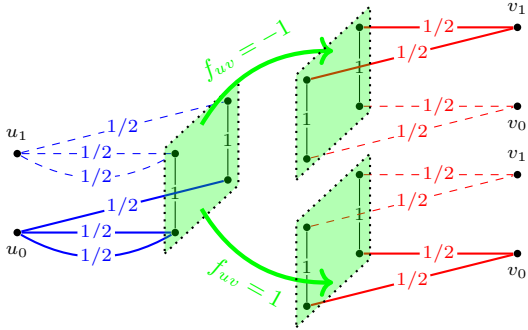

\begin{lemma}[Mirroring argument]\label{lem:mirroring}
    If $s \in \mathcal{S}_{\infty}$ then $F(s) = 0$.
\end{lemma}

\begin{proof}
    Fix $s\in \mathcal{S}_{\infty}$. To lighten notation we write $\pi=\pi_s$. The claim is trivial if $\pi(u)$ or $\pi(v)$ belongs to $\pi(t)$, so suppose not. Note that $\pi(u) \neq \pi(v)$ as otherwise $d(s) = 0$. 
    
    Let $U$ be the set of vertices in $G(s)$ that have a vertical-free path to $u$, $E_u$ be the set of edges of $G(s)$ that have at least one endpoint in $U$ and $E_v=E(s)\setminus E_u$. The partition $\brac{E_u, E_v}$ has the properties that any vertex that is an endpoint both of an edge in $E_u$ and an edge in $E_v$ must belong to $\pi(t)$, that every edge incident to $\pi(u)$ belongs to $E_u$, and that every edge incident to $\pi(v)$ belongs to $E_v$.
    
    Let $\omega$ be a sample from $\mathbb{Q}^{\pi(t)}_{G(s)}$, and let $f_{uv}$ be the random variable
    \[f_{uv}(\omega) = \indi{\pi(u)_0 \leftrightarrow \pi(v)_0} + \indi{\pi(u)_1 \leftrightarrow \pi(v)_1} - \indi{\pi(u)_0 \leftrightarrow \pi(v)_1} - \indi{\pi(u)_1 \leftrightarrow \pi(v)_0},\]
    where connections are defined in $\omega$. If we define $\omega'$ by flipping the value of each edge in $E_v$ (i.e., replacing each up edge with a down edge and vice versa) then we see that $\omega'$ has the same distribution as $\omega$ and that $f_{uv}(\omega')=-f_{uv}(\omega)$ as depicted in Figure \ref{fig:mirroring}: If there is a path from $u_0$ to $v_0$ in $\omega$ then there is a path from $u_0$ to $v_1$ in $\omega'$ and vice versa. It follows that $f_{uv}(\omega)$ has expectation zero, which is equivalent to the claim by symmetry.
\end{proof}

Next we establish a uniform positive lower bound on $\mathcal{S}_{\leqslant 1}$.

\begin{lemma}\label{lem:Splus}
    If $s\in \mathcal{S}_{\leqslant 1}$ then $F(s) \geqslant 2^{1-\abs{s_1}} \geqslant 2^{1-\abs{E}}$.
\end{lemma}

\begin{proof}
    Fix $s\in \mathcal{S}_{\leqslant 1}$ and write $\pi=\pi_s$. Let $\omega$ and $f_{uv}$ be as above. First observe that $f_{uv}(\omega) \geqslant 0$ deterministically.
    Indeed, since either $u_0 \leftrightarrow v_0$ or $u_1 \leftrightarrow v_1$, we have that the sum of the first two indicators is at least $1$. Thus, for $f_{uv}$ to be negative we would require both of the latter two indicators to be $1$, but in this case at least one of the two events $\{\pi(u)_0 \leftrightarrow \pi(v)_1 \leftrightarrow \pi(v)_0 \leftrightarrow \pi(u)_1\}$ or $\{\pi(v)_0 \leftrightarrow \pi(u)_1 \leftrightarrow \pi(u)_0 \leftrightarrow \pi(v)_1\}$ must hold, so that the vertices $\{\pi(u)_0,\pi(u)_1,\pi(v)_0,\pi(v)_1\}$ all belong to the same connected component of $\omega$ and $f_{uv}(\omega)=0$.
    
    To conclude, we note that if every edge of $G(s)$ is up or every edge of $G(s)$ is down then $f_{uv} = 2$ (if $d(s) = 0$) or $f_{uv} = 1$ (if $d(s) = 1$), and this happens with probability $2^{1-\abs{s_1}}$. \qedhere
\end{proof}

The last ingredient of the proof is the following lemma, which shows that the contribution of $\mathcal{S}_{\geqslant 2}$ to \eqref{eq:Linusson_expansion} is negligible compared to that of $\mathcal{S}_{\leqslant 1}$ when $p$ is close to $1$. Note that $\mathbb{P}_p(S =s \mid T=t)$ does not depend on the value of $t$ for each $s\in \mathcal{S}$, but that the sets $\mathcal{S}_{\infty}$, $\mathcal{S}_{\leqslant 1}$, and $\mathcal{S}_{\geqslant 2}$ depend on $T$ so that e.g.\ $\mathbb{P}_p(S \in \mathcal{S}_{\geqslant 2} \mid T=t)$ may depend on $t$.

\begin{lemma}\label{lem:Sminus}
    The inequality $\probm{\mathrm{bb}}{p}{S\in \mathcal{S}_{\geqslant 2} \mid T=t} \leqslant 2^{-\abs{E}} \cdot \probm{\mathrm{bb}}{p}{S\in \mathcal{S}_{\leqslant 1}\mid T=t}$ holds for every $p \geqslant 1-2^{-\abs{E}/2-2}$, with strict inequality if $\probm{\mathrm{bb}}{p}{S\in \mathcal{S}_{\leqslant 1}\mid T=t}$ is positive.
\end{lemma}

\begin{proof}
    If $s\in \mathcal{S}_{\geqslant 2}$ then there exists a set $A = \brac{e_1, \ldots, e_k} \subseteq s_1$ such that $\pi_s(A)$ is a vertical-free path in $G(s)$ connecting $\pi_s(u)$ and $\pi_s(v)$ with $k\geqslant 2$. As such, if we define a tripartition $s'$ by $s' = \brap{s_0, s_1 \setminus A, s_2' = s_2 \cup A}$ then $s' \in \mathcal{S}_{\leqslant 1}$. Note that the construction of $s'$ from $s$ depends on the choice of the set $A$, which may not be unique.
    It follows that every $s\in \mathcal{S}_{\geqslant 2}$ can be written (perhaps non-uniquely) in the form $(s'_0,s'_1 \cup A, s'_2 \setminus A)$ where $s'\in \mathcal{S}_{\leqslant 1}$ and $A \subseteq s'_2$ is non-empty. Hence
    \begin{align}
    \probm{\mathrm{bb}}{p}{S\in \mathcal{S}_{\geqslant 2} \mid T=t}
    &\leqslant \sum_{s'\in \mathcal{S}_{\leqslant 1}} \sum_{k\geqslant 2}\probm{\mathrm{bb}}{p}{S =s'} \binom{\abs{s_2'}}{k} \brap{\frac{2p(1-p)}{p^2}}^k \nonumber\\
    &\leqslant \bras{\sum_{k\geqslant 2} \binom{\abs{E}}{k} \brap{\frac{2p(1-p)}{p^2}}^k} \bras{\sum_{s'\in \mathcal{S}_{\leqslant 1}} \probm{\mathrm{bb}}{p}{S =s'}} \nonumber\\
    &= \bras{f\brap{\abs{E}\frac{2(1-p)}{p}} - f(0)} \probm{\mathrm{bb}}{p}{S\in \mathcal{S}_{\leqslant 1}\mid T=t}.
    \label{eq:SminusSplusratio}
    \end{align}
    where $f(x) = \brap{1+\frac{x}{\abs{E}}}^{\abs{E}}-x$. Observe that $f'(0) = 0$ and $f''(x) = \brap{1+x}^{\abs{E}-2}$ for $x \geqslant 0$ so if
    \begin{equation}
    \label{eq:calculus}
    \frac{2(1-p)}{p} \leqslant \frac{1}{\abs{E}} \qquad \text{then} \qquad \bras{f\brap{\frac{2(1-p)}{p}} - f(0)} \leqslant \frac{1}{2}\brap{\frac{2(1-p)}{p}}^2\brap{1+\frac{1}{\abs{E}}}^{\abs{E}-2} < 2e\brap{\frac{1-p}{p}}^2.
    \end{equation}
    It follows from \eqref{eq:SminusSplusratio} and \eqref{eq:calculus} that we have the implications
    \begin{align}
    p \geqslant \frac{1}{1+\brap{2e}^{-1/2} \cdot 2^{-\abs{E}/2}} \quad &\Rightarrow \quad \brap{\bras{f\brap{\frac{2(1-p)}{p}} - f(0)} < 2e\brap{\frac{1-p}{p}}^2 \leqslant 2^{-\abs{E}} }
    \nonumber\\
    &\Rightarrow \quad \brap{\probm{\mathrm{bb}}{p}{S\in \mathcal{S}_{\geqslant 2}\mid T=t} \leqslant 2^{-\abs{E}} \probm{\mathrm{bb}}{p}{S\in \mathcal{S}_{\leqslant 1}\mid T=t}},\end{align}
    and the claim follows by verifying that the elementary inequality $\frac{1}{1+\brap{2e}^{-1/2} x} \leqslant 1-\frac{1}{4}x$ holds for every $x\in [0,1]$. Moreover, strict inequality is easily seen to hold when $\probm{\mathrm{bb}}{p}{S\in \mathcal{S}_{\leqslant 1}\mid T=t}$ is positive.
\end{proof}

We are now ready to complete the proof of Theorem~\ref{thm:mainprime} and hence of Theorem~\ref{thm:main}.

\begin{proof}[Proof of Theorem~\ref{thm:mainprime}]
Substituting the estimates of Lemmas \ref{lem:mirroring}, \ref{lem:Sminus}, and \ref{lem:Splus} into the expansion \eqref{eq:Linusson_expansion} yields that
\begin{align}
    &\probm{\mathrm{bb}}{p}{u_0 \leftrightarrow v_0 \mid T=t} - \probm{\mathrm{bb}}{p}{u_0 \leftrightarrow v_1 \mid T=t} \nonumber\\
    &\hspace{5cm}= 
    \hfsetfillcolor{lightyellow}\tikzmarkin{c1}(0,-1.3)(0,0.6) \underbrace{\sum_{s \in \mathcal{S}_{\infty}} \probm{\mathrm{bb}}{p}{S=s}F(s)}_{\hspace{0.5em}=0^{\phantom{2}}} \tikzmarkend{c1} + 
    \hfsetfillcolor{lightgreen}\tikzmarkin{c2}(0,-1.3)(0,0.6) \underbrace{\sum_{s \in \mathcal{S}_{\leqslant 1}} \probm{\mathrm{bb}}{p}{S=s}F(s)}_{\geqslant 2^{1-|E|} \cdot \probm{\mathrm{bb}}{p}{S \in \mathcal{S}_{\leqslant 1}}} \tikzmarkend{c2} +
    \hfsetfillcolor{lightred}\tikzmarkin{c3}(0,-1.3)(0,0.6) \underbrace{\sum_{s \in \mathcal{S}_{\geqslant 2}} \probm{\mathrm{bb}}{p}{S=s}F(s)}_{\geqslant -2 \cdot \probm{\mathrm{bb}}{p}{S \in \mathcal{S}_{\geqslant 2}}} \tikzmarkend{c3} \phantom{\Biggr]}
    \nonumber\\
    &\hspace{5cm}\geqslant 2 \brap{2^{-|E|}\probm{\mathrm{bb}}{p}{S \in \mathcal{S}_{\leqslant 1} \mid T=t} - \probm{\mathrm{bb}}{p}{S \in \mathcal{S}_{\geqslant 2} \mid T=t}} \geqslant 0,\phantom{\Biggr]}
\label{eq:final}
\end{align}
where we have used Lemma \ref{lem:Sminus} in the final inequality. 
Lemma \ref{lem:Sminus} also implies that the final inequality is strict whenever $\probm{\mathrm{bb}}{p}{S \in \mathcal{S}_{\leqslant 1} \mid T=t}>0$, which is easily seen to be the case if and only if there is a vertical-free path between $u$ and $v$ in $G$ as claimed.
\end{proof}

\begin{remark}
Using sharper inequality $\frac{1}{1+\brap{4e}^{-1/2} x} \leqslant 1-\frac{1}{4}x$ for every $x\in [0,2/3]$, the analysis of Lemma \ref{lem:Sminus} yields that $\probm{\mathrm{bb}}{p}{S \in \mathcal{S}_{\geqslant 2}} \leqslant 2^{-\abs{E}-1}\probm{\mathrm{bb}}{p}{S \in \mathcal{S}_{\leqslant 1}}$, so that \eqref{eq:final} yields that
\begin{equation}
    \probm{\mathrm{bb}}{p}{u_0 \leftrightarrow v_0 \mid T=t} - \probm{\mathrm{bb}}{p}{u_0 \leftrightarrow v_1 \mid T=t} \geqslant 2^{-\abs{E}} \cdot \probm{\mathrm{bb}}{p}{S \in \mathcal{S}_{\leqslant 1} \mid T=t}.
\end{equation}
If $u$ and $v$ are connected in $V \setminus t$ then $\probm{\mathrm{bb}}{p}{S \in \mathcal{S}_{\leqslant 1} \mid T=t}$ is positive, and indeed there must exist a tripartition $s\in \mathcal{S}_+$ with $s_0=\emptyset$. Indeed, to construct such an $s$, take a simple vertical-free path in $G$ from $u$ to $v$, include all edges in $s_2$, and include every other edge of $G$ in $s_1$. Since $p\geqslant 2/3$, it follows that the probability $S$ belongs to $\mathcal{S}_{\leqslant 1}$ is at least $(2(1-p)p)^{\abs{E}}$  and hence that if $u$ and $v$ are connected in $V \setminus t$ then
\begin{equation}
\mathbb{P}_p^\mathrm{bb}(u_0 \leftrightarrow v_0 \mid T=t) - \mathbb{P}_p^\mathrm{bb}(u_0 \leftrightarrow v_1 \mid T=t) \geqslant p^{\abs{E}}(1-p)^{\abs{E}}
\end{equation}
for every $p\geqslant 1-2^{-\abs{E}/2-2}$.
\end{remark}

\bibliographystyle{abbrv}
\bibliography{bunkbed.bib}

\begin{thebibliography}{10}

\bibitem{MR2268229}
J.~van~den Berg, O.~H\"{a}ggstr\"{o}m, and J.~Kahn.
\newblock Some conditional correlation inequalities for percolation and related
  processes.
\newblock {\em Random Structures Algorithms}, 29(4):417--435, 2006.

\bibitem{MR1825144}
J.~van~den~Berg and J.~Kahn.
\newblock A correlation inequality for connection events in percolation.
\newblock {\em Ann. Probab.}, 29(1):123--126, 2001.

\bibitem{deBuyer2016}
P.~de~Buyer.
\newblock A proof of the bunkbed conjecture on the complete graph for $p= 1/2$.
\newblock {\em arXiv preprint arXiv:1604.08439}, 2016.

\bibitem{deBuyer2018}
P.~de~Buyer.
\newblock A proof of the bunkbed conjecture on the complete graph for $p\geq
  1/2$.
\newblock {\em arXiv preprint arXiv:1802.04694}, 2018.

\bibitem{MR1707339}
G.~Grimmett.
\newblock {\em Percolation}, volume 321 of {\em Grundlehren der Mathematischen
  Wissenschaften [Fundamental Principles of Mathematical Sciences]}.
\newblock Springer-Verlag, Berlin, second edition, 1999.

\bibitem{MR3751350}
G.~Grimmett.
\newblock {\em Probability on graphs}, volume~8 of {\em Institute of
  Mathematical Statistics Textbooks}.
\newblock Cambridge University Press, Cambridge, 2018.
\newblock Random processes on graphs and lattices, Second edition of [
  MR2723356].

\bibitem{MR1680084}
O.~H\"{a}ggstr\"{o}m.
\newblock On a conjecture of {B}ollob\'{a}s and {B}rightwell concerning random
  walks on product graphs.
\newblock {\em Combin. Probab. Comput.}, 7(4):397--401, 1998.

\bibitem{haggstrom2003probability}
O.~H{\"a}ggstr{\"o}m.
\newblock Probability on bunkbed graphs.
\newblock In {\em Proceedings of FPSAC}, volume~3, 2003.

\bibitem{MR1043524}
T.~Hara and G.~Slade.
\newblock Mean-field critical behaviour for percolation in high dimensions.
\newblock {\em Comm. Math. Phys.}, 128(2):333--391, 1990.

\bibitem{MR1959796}
T.~Hara, R.~van~der Hofstad, and G.~Slade.
\newblock Critical two-point functions and the lace expansion for spread-out
  high-dimensional percolation and related models.
\newblock {\em Ann. Probab.}, 31(1):349--408, 2003.

\bibitem{MR3886521}
P.~van~Hintum and P.~Lammers.
\newblock The bunkbed conjecture on the complete graph.
\newblock {\em European J. Combin.}, 76:175--177, 2019.

\bibitem{Leander}
M.~Leander.
\newblock On the bunkbed conjecture.
\newblock {\em Sj\"alvst\"andiga Arbeten I Matematik (Independent Works In
  Mathematics)}, 2009.
\newblock Matematiska Institutionen, Stockholms Universitet.

\bibitem{MR2745680}
S.~Linusson.
\newblock On percolation and the bunkbed conjecture.
\newblock {\em Combin. Probab. Comput.}, 20(1):103--117, 2011.

\bibitem{MR4015661}
S.~Linusson.
\newblock Erratum to `{O}n percolation and the bunkbed conjecture'.
\newblock {\em Combin. Probab. Comput.}, 28(6):917--918, 2019.

\bibitem{MR676312}
A.~Messager and S.~Miracle-Sole.
\newblock Correlation functions and boundary conditions in the {I}sing
  ferromagnet.
\newblock {\em J. Statist. Phys.}, 17(4):245--262, 1977.

\bibitem{rudzinski2016equivalent}
J.~Rudzinski and C.~Smyth.
\newblock Equivalent formulations of the bunk bed conjecture.
\newblock {\em The North Carolina Journal of Mathematics and Statistics},
  2:23--28, 2016.





\end{thebibliography}

\end{document}